\documentclass[a4paper,leqno,10pt]{amsart}

\usepackage{epsf,graphicx,epsfig}
\usepackage{amscd}
\usepackage{amsmath,latexsym,amssymb,amsthm}
\usepackage[nospace,noadjust]{cite}
\usepackage{textcomp}
\usepackage{setspace,cite}
\usepackage{lscape,fancyhdr,fancybox}
\usepackage{stmaryrd}
\usepackage[all,cmtip]{xy}
\usepackage{tikz}
\usepackage{cancel}
\usetikzlibrary{shapes,arrows,decorations.markings}
\usepackage{graphicx}

\usepackage{color}
\usepackage{url}
\usepackage{enumerate}
\usepackage[mathscr]{euscript}

  \theoremstyle{plain}

\swapnumbers
    \newtheorem{thm}{Theorem}[section]
    \newtheorem{prop}[thm]{Proposition}

    \newtheorem{subsec}[thm]{}
\theoremstyle{definition}
    \newtheorem{defn}[thm]{Definition}
        \newtheorem{remark}[thm]{Remark}
\theoremstyle{remark}

\title{}
\author{}
\date{}
\usepackage{amssymb}

\usepackage{hyperref}
\hypersetup{
	colorlinks,
	citecolor=blue,
	filecolor=black,
	linkcolor=blue,
	urlcolor=black
}

\begin{document}
\title{Homotopy $G$-algebra structure on the cochain complex of hom-type algebras}

\author{Apurba Das}

\curraddr{Indian Statistical Institute, Kolkata, Stat-Math Unit, 203 BT Road, Kolkata-700108, India}
\email{apurbadas348@gmail.com}

\subjclass[2010]{16E40, 17A30, 17A99.}
\keywords{Gerstenhaber algebra, homotopy $G$-algebras, hom-algebras.}

\begin{abstract}
A hom-associative algebra is an algebra whose associativity is twisted by an algebra homomorphism. We show that the Hochschild type cochain complex of a hom-associative algebra carries a homotopy $G$-algebra structure. As a consequence, we get a Gerstenhaber algebra structure on the cohomology of a hom-associative algebra. We also find similar results for hom-dialgebras. 
\end{abstract}

\noindent

\thispagestyle{empty}

\maketitle


\vspace{0.2cm}
\section{Introduction}
In \cite{gers} Gerstenhaber showed that the Hochschild cohomology $H^\bullet(A,A)$ of an associative algebra $A$ carries a certain algebraic structure. This algebraic structure is now known as Gerstenhaber algebra. A Gerstenhaber algebra  is a graded commutative associative algebra together with a degree $-1$ graded Lie bracket which are compatible in the sense of a suitable Leibniz rule. 
An alternative proof of the same fact has been carried out by Gerstenhaber and Voronov \cite{gers-voro}. More precisely, they prove a more general statement that the Hochschild complex $C^\bullet (A,A)$ carries a homotopy $G$-algebra structure. 

A homotopy $G$-algebra is a brace algebra $(\mathcal{O} = \oplus \mathcal{O}(n), \{-\} \{-, \ldots, -\})$ together with a differential graded associative algebra structure on $\mathcal{O}$ satisfying some compatibility conditions \cite{gers-voro}. The brace algebra structure on $C^\bullet (A,A)$ is given by the classical braces introduced by Getzler-Jones \cite{getz-jon}, the  differential graded associative algebra structure on $C^\bullet (A,A)$ is given by the ususal cup product and the Hochschild coboundary (up to some signs). In \cite{gers-voro} the authors showed that the existence of the homotopy $G$-algebra structure on $C^\bullet (A, A)$ is based on the non-symmetric endomorphism operad structure on $C^\bullet (A, A)$ together with a multiplication on that operad. The same idea has been used to define a homotopy $G$-algebra structure on the dialgebra complex $CY^\bullet (D,D)$ of a dialgebra $D$ \cite{maj-muk}.

In this paper, we deal with certain type of algebras, called hom-type algebras.
 In these algebras, the identities defining the structures are twisted by homomorphisms.
 Recently, hom-type algebras have been studied by many authors.
The notion of hom-Lie algebras was first introduced by Hartwig, Larsson and Silvestrov \cite{hls}.
Hom-Lie algebras appeared in examples of $q$-deformations of the Witt and Virasoro algebras. Other type of algebras (e.g. associative, Leibniz, Poisson, Hopf,...) twisted by homomorphisms have also been studied. See \cite{makh-sil, makh-sil2} (and references there in) for more details.
Our main objective in this paper is the notion of hom-associative algebra introduced by Makhlouf and Silvestrov \cite{makh-sil}. A hom-associative algebra is an algebra $(A, \mu)$ whose associativity is twisted by an algebra homomorphism $\alpha : A \rightarrow A$ (cf. Definition \ref{hom-ass-defn}).  When $\alpha$ is the identity map, we recover the classical notion of associative algebras as a subclass. 

In \cite{amm-ej-makh, makh-sil2} the authors studied the formal one-parameter deformation of hom-associative algebras and introduce a Hochschild type cohomology theory for hom-associative algebras. Given a hom-associative algebra $(A, \mu, \alpha)$, its $n$-th cochain group $C^n_\alpha (A, A)$ consists of multiliear maps $f : A^{\otimes n} \rightarrow A$ which satisfies
$ \alpha \circ f = f \circ \alpha^{\otimes n} $ and the coboundary operator $\delta_\alpha$ is similar to the Hochschild coboundary but suitably twisted by $\alpha$. In \cite{amm-ej-makh} the authors also introduce a degree $-1$ graded Lie bracket $[-,-]_\alpha$ on the cochain groups $C^\bullet_\alpha (A, A)$ which passes on to cohomology. In \cite{das} the present author defines a cup product $\cup_\alpha$ on the cochain groups $C^\bullet_\alpha (A, A)$ and showed that it induces a graded commutative, associative product on the cohomology $H^\bullet_\alpha (A, A)$. Moreover, it was shown that the induced structures on the cohomology $H^\bullet_\alpha (A,A)$ makes it a Gerstenhaber algebra.

In this paper, we follow the method of Gerstenhaber and Voronov \cite{gers-voro}. We show that the cochain complex $C^\bullet_\alpha (A, A)$ carries a non-symmetric operad structure. This operad structure is similar to the endomorphism operad on $A$, however, twisted by $\alpha$. Moreover, the multiplication defining the hom-associative structure gives a multiplication in the above operad. Hence, by a result of  \cite{gers-voro}, it follows that the cochain complex $C^\bullet_\alpha (A,A)$ carries a homotopy $G$-algebra structure. As a consequence, we get a Gerstenhaber algebra structure on cohomology. This gives an alternative approach of the same result proved by the author \cite{das}.

The notion of (diassociative) dialgebras was introduced by Loday as a generalization of associative algebras \cite{lod}. The hom-analogue of a dialgebra is known as a hom-dialgebra \cite{yau2}. We discuss the above results for hom-dialgebras. Given a hom-dialgebra $D$, 
we show that the cochain complex $CY^\bullet_\alpha (D,D)$ defining the cohomology of a hom-dialgebra
carries a non-symmetric operad structure. Moreover, the operations defining the hom-dialgebra structure induces a multiplication on the operad. Hence, we conclude that the cochain complex $CY^\bullet_\alpha (D,D)$ inherits a homotopy $G$-algebra structure and the corresponding cohomology $HY^\bullet_\alpha (D,D)$ carries a Gerstenhaber algebra structure.

In Section \ref{sec2}, we recall some basic preliminaries on operads, braces and homotopy $G$-algebras. In Section \ref{sec3}, we first revise hom-associative algebras and prove our results for hom-associative algebras. Finally, in Section \ref{sec4}, we deal with hom-dialgebras.

\section{Preliminaries}\label{sec2}
In this section, we recall some basic definitions. See \cite{gers-voro, getz-jon} for more details.
\begin{defn}
	A  non-symmetric operad (non-$\sum$ operad in short) in the category of vector spaces is a collection of vector spaces $\{ \mathcal{O} (k) |~ k \geq 1 \}$ together with compositions
	\begin{align*}
	\gamma : \mathcal{O} (k) \otimes \mathcal{O} (n_1) \otimes \cdots \otimes \mathcal{O} (n_k) &\rightarrow \mathcal{O} (n_1 + \cdots + n_k) \\
	f \otimes g_1 \otimes \cdots \otimes g_k  &\mapsto  \gamma (f ; g_1 , \ldots, g_k )
	\end{align*}
	 which is associative in the sense that 
\begin{align*}
&\gamma \big(  \gamma (f; g_1 , \ldots, g_k); h_1 , \ldots, h_{n_1 + \cdots + n_k}    \big) \\
&= \gamma \big( f; ~\gamma (g_1 ; h_1, \ldots, h_{n_1}), ~ \gamma ( g_2 ; h_{n_1 + 1}, \ldots, h_{n_1 + n_2}) , \ldots, ~\gamma (g_k ; h_{n_1 + \cdots + n_{k-1}+1}, \ldots, h_{n_1 + \cdots + n_k})  \big)
\end{align*}	 
and there is an identity element id$ \in \mathcal{O} (1)$ such that
\begin{center}
$\gamma (f ; \underbrace{\text{id}, \ldots, \text{id}}_{k \text{ times}} ) = f = \gamma (\text{id}; f)$, ~~ \text{ for } $f \in \mathcal{O} (k)$. 
\end{center}
\end{defn}


A non-$\sum$ operad can also be described by compositions (called partial compositions)
$$\circ_i : \mathcal{O}(m) \otimes \mathcal{O}(n) \rightarrow \mathcal{O}(m+n-1), \quad 1 \leq i \leq m$$
satisfying
$$ \begin{cases} (f \circ_i g) \circ_{i+j-1} h = f \circ_i (g \circ_j h), \quad &\mbox{~~~ for } 1 \leq i \leq m, ~1 \leq j \leq n, \\ (f \circ_i g) \circ_{j+n-1} h = (f \circ_j h) \circ_i g, \quad  & \mbox{~~~ for } 1 \leq i < j \leq m, \end{cases}$$
for $f \in \mathcal{O}(m), ~ g \in \mathcal{O}(n), ~ h \in \mathcal{O}(p),$
and an identity element satisfying $ f \circ_i \text{id} = f =\text{id}\circ_1 f,$ for all $f \in \mathcal{O}(k)$ and $1 \leq i \leq m$. The two definitions of non-$\sum$ operad are related by
\begin{align}
f \circ_i g =~& \gamma (f ;~ \overbrace{\text{id}, \ldots, \text{id} , \underbrace{g}_{i\text{-th place}}, \text{id}, \ldots, \text{id}}^{m\text{-tuple}}),~~~ \quad \text{ for }f \in \mathcal{O}(m),\label{eqn-1}\\
\gamma (f ; g_1, \ldots, g_k) =~&   (\cdots ((f \circ_k g_k) \circ_{k-1} g_{k-1}) \cdots ) \circ_1 g_1 , \quad \text{ for }f \in \mathcal{O}(k). \label{eqn-2}
\end{align}

A toy example of an operad is given by the endomorphisms of a vector space.
Let $A$ be a vector space and define $\mathcal{O}(k) = Hom(A^{\otimes k} , A)$, for $k \geq 1$. The compositions $\gamma$ are substitution of the values of $k$ operations in a $k$-ary operation as inputs. 

\medskip

Next, consider the graded vector space $\mathcal{O} = \oplus_{k \geq 1} \mathcal{O}(k)$ of an operad. If $f \in \mathcal{O}(n)$, we define deg $f = n$ and 
 $|f| = n-1$. We use the same notation for any graded vector space as well. Consider the braces
\begin{center}
$\{ f \} \{ g_1, \ldots, g_n\} := \sum (-1)^\epsilon ~ \gamma (f ; \text{id}, \ldots, \text{id}, g_1, \text{id}, \ldots, \text{id}, g_n, \text{id}, \ldots, \text{id})$
\end{center}
where the summation runs over all possible substitutions of $g_1, \ldots, g_n$ into $f$ in the pescribed order and $\epsilon := \sum_{p=1}^{n} |g_p| i_p$, $i_p$ being the total number of inputs in front of $g_p$. The multilinear braces $\{f\} \{ g_1, \ldots, g_n\}$ are homogeneous of degree $-n$. 
Moreover, they satisfy the following identities

\noindent {\sf $\bullet$ Higher pre-Jacobi identities:}
\begin{align*}
&\{f\} \{ g_1, \ldots, g_m\} \{ h_1, \ldots, h_n \}\\
&= \sum_{0 \leq i_1 \leq \cdots \leq i_m \leq n}^{} (-1)^\epsilon ~\{f\} \{ h_1, \ldots, h_{i_1}, \{ g_1 \} \{h_{i_1 +1}, \ldots, h_{j_1} \}, h_{j_1+1}, \ldots, h_{i_m},\\
&\hspace*{8cm}~~~ \{g_m\} \{ h_{i_m +1}, \ldots, h_{j_m} \}, h_{j_m+1}, \ldots, h_n  \},
\end{align*}
where $\epsilon := \sum_{p=1}^{m} \big( |g_p| \sum_{q=1}^{i_p} |h_q|  \big).$

One also assume the following conventions in an operad:
\begin{center}
$\{f\}\{~ \} := f   ~~~\text{ and }~~~ f \circ g := \{f\} \{ g\}.$
\end{center}

\begin{remark}
The higher pre-Jacobi identities implies that
\begin{align}\label{lie-brckt}
[f,g] = f \circ g - (-1)^{|f||g|} g \circ f, ~~~ \text{ for } f , g \in \mathcal{O},
\end{align}
defines a degree $-1$ graded Lie bracket on $\mathcal{O}$.
\end{remark}

\begin{defn}
	A multiplication on an operad $\mathcal{O}$ is an element $m \in \mathcal{O} (2)$ such that $m \circ m = 0$.
\end{defn}

If $m$ is a multiplication on an operad $\mathcal{O}$, then the dot product
\begin{center}
$f \cdot g = (-1)^{|f| + 1} \{ m \} \{ f, g\}, ~~~ f , g \in \mathcal{O},$
\end{center}
defines a graded associative algebra structure on $\mathcal{O}$. Moreover, the degree one map $d : \mathcal{O} \rightarrow \mathcal{O}$, $f \mapsto m \circ f - (-1)^{|f|} f \circ m$ is a differential on $\mathcal{O}$ and the triple $(\mathcal{O}, \cdot, d)$ is a differential graded associative algebra \cite{gers-voro}. Moreover, the following identities are hold:

\noindent {\sf $\bullet$ Distributivity:}
$$\{f \cdot g\} \{ h_1, \ldots, h_n \} = \sum_{k=0}^{n}  (-1)^\epsilon ~ (\{f \}\{ h_1, \ldots, h_k\}) \cdot (\{g\}\{ h_{k+1}, \ldots, h_n \}),~~ \text{ where } \epsilon = |g| \sum_{p=1}^{k} |h_p|.$$


\noindent {\sf $\bullet$ Higher homotopies:}
\begin{align*}
&	d ( \{f\}\{ g_1, \ldots, g_{n+1} \}  ) - \{df\} \{ g_1, \ldots, g_{n+1} \} - (-1)^{|f|} \sum_{i=1}^{n+1} (-1)^{|g_1| + \cdots + |g_{i-1}|} \{f\} \{ g_1, \ldots, dg_i, \ldots, g_{n+1} \}\\
	=& (-1)^{|f||g_1| +1} g_1 \cdot (\{f\} \{ g_2, \ldots, g_{n+1}\}) + (-1)^{|f|} \sum_{i=1}^{n} (-1)^{|g_1| + \cdots+|g_{i-1}|} \{f\} \{ g_1, \ldots, g_i \cdot g_{i+1}, \ldots, g_{n+1} \}  \\
	& - \{f\} \{ g_1, \ldots, g_n \} \cdot g_{n+1}.
\end{align*}

Summarizing the properties of braces and multiplications on an operad, one get the following algebraic structures \cite{gers-voro, getz-jon}.

\begin{defn}
A  brace algebra is a graded vector space $\mathcal{O} = \mathcal{O}(n)$ together with a collection of braces $\{f\} \{g_1, \ldots, g_n \}$ of degree $-n$ satisfying the higher pre-Jacobi identities.
\end{defn}

A brace algebra as above may be denoted by $(\mathcal{O} = \mathcal{O}(n) , ~\{-\}\{-, \ldots, -\})$. 

\begin{defn}
	A  homotopy G-algebra is a brace algebra $(\mathcal{O} = \mathcal{O}(n), ~\{-\}\{-, \ldots, -\})$ endowed with a differential graded associative algebra structure $(\mathcal{O} = \mathcal{O}(n), \cdot, d )$ satisfying the distributivity and higher homotopies. A homotopy $G$-algebra is denoted by $(\mathcal{O} = \mathcal{O}(n), ~\{-\}\{-, \ldots, -\}), \cdot, d)$.
\end{defn}

As a summary, we get the following \cite{gers-voro}.
\begin{thm}\label{gers-voro-thm}
A multiplication on an operad $\mathcal{O}$ defines the structure of a homotopy $G$-algebra on $\mathcal{O} = \oplus \mathcal{O}(n).$
\end{thm}

Next, we recall Gerstenhaber algebras ($G$-algebras in short).

\begin{defn}
A (left) Gerstenhaber algebra is a graded commutative associative algebra $(\mathcal{A} = \oplus \mathcal{A}^i, \cdot)$ together with a degree $-1$ graded Lie bracket $[-,-]$ on $\mathcal{A}$ satisfying the following Leibniz rule
$$[a, b \cdot c] = [a, b] \cdot c + (-1)^{|a|(|b| + 1)} b \cdot [a, c],$$
for all homogeneous elements $a, b, c \in \mathcal{A}.$
\end{defn}

\begin{remark}\label{gers-voro-rem}
 Given a homotopy $G$-algebra $(\mathcal{O} = \mathcal{O}(n), ~\{-\}\{-, \ldots, -\}), \cdot, d)$, the product $\cdot$ induces a graded commutative associative product $\cdot$ on the cohomology $H^\bullet (\mathcal{O}, d)$. The degree $-1$ graded Lie bracket as defined in (\ref{lie-brckt}) also passes on to the cohomology $H^\bullet (\mathcal{O}, d)$. Moreover, the induced product and the bracket on the cohomology satisfy the graded Leibniz rule to becomes a Gerstenhaber algebra \cite{gers-voro}.
\end{remark}

\section{Hom-associative algebras}\label{sec3}
In this section, we first recall hom-associative algebras and their Hochschild cohomology.  Then we show that the Hochschild complex of hom-associative algebras carries a natural operad structure together with a multiplication. Finally, we deduce a Gerstenhaber algebra structure on the cohomology.
\begin{defn}\label{hom-ass-defn}
	A hom-associative algebra over $\mathbb{K}$ is a triple $(A, \mu, \alpha)$ consists of a $\mathbb{K}$-vector space $A$ together with a $\mathbb{K}$-bilinear map
	$\mu : A \times A \rightarrow A$ and a $\mathbb{K}$-linear map $\alpha : A \rightarrow A$ satisfying $\alpha (\mu (a,b)) = \mu(\alpha (a), \alpha (b))$ and
	\begin{align}\label{hom-ass-cond}
	\mu ( \alpha (a) , \mu ( b , c) ) = \mu ( \mu (a , b) , \alpha (c)), ~~ \text{ for all } a, b, c \in A.
	\end{align}
\end{defn}

In \cite{amm-ej-makh} the authors called such a hom-associative algebra 'multiplicative'. By a hom-associative algebra, they mean a triple $(A, \mu, \alpha)$ of a vector space $A$, a bilinear map $\mu : A \times A \rightarrow A$ and a linear map $\alpha : A \rightarrow A$ satisfying condition (\ref{hom-ass-cond}). See \cite{amm-ej-makh, makh-sil} for examples of hom-associative algebras.

When $\alpha =$ identity, in any case, one gets the definition of a classical associative algebra.
Next, we recall the definition of Hochschild type cohomology for hom-associative algebras. Like classical case, these cohomology theory controls the deformation of hom-associative algebras \cite{amm-ej-makh}.

Let $(A, \mu, \alpha)$ be a hom-associative algebra. For each $n \geq 1$, we define a $\mathbb{K}$-vector space $C^n_\alpha (A,A) $ consisting of all multilinear maps $f : A^{\otimes n} \rightarrow A$ satisfying $\alpha \circ f = f \circ \alpha^{\otimes n}$, that is,
$$ (\alpha \circ f) (a_1, \ldots, a_n) =  f ( \alpha(a_1), \ldots , \alpha (a_n )), ~~\text{ for all }a_i \in A .$$
Define $\delta_\alpha : C^n_\alpha (A, A) \rightarrow C^{n+1}_\alpha (A, A)$ by the following
\begin{align*}
(\delta_\alpha f) (a_1, a_2, \ldots, a_{n+1}) =~& \mu \big(\alpha^{n-1}(a_1) , f(a_2, \ldots , a_{n+1}) \big) \\
~& + \sum_{i=1}^{n} (-1)^i~ f \big( \alpha(a_1), \ldots, \alpha (a_{i-1}), \mu (a_i , a_{i+1}), \alpha (a_{i+2}), \ldots, \alpha (a_{n+1}) \big) \\
~& + (-1)^{n+1} \mu (f (a_1, \ldots, a_n) , \alpha^{n-1} (a_{n+1})).
\end{align*}
Then we have $\delta^2_\alpha = 0$. The cohomology of this complex is called the Hochschild cohomology of the hom-associative algebra $(A, \mu, \alpha).$ The cohomology groups are denoted by 
$H^n_\alpha (A, A),~ n \geq 2.$ When $\alpha =$ identity, one recovers the classical Hochschild cohomology of associative algebras.

\medskip

\noindent {\bf Operad structure:} Let $A$ be a vector space and $\alpha : A \rightarrow A$ be a linear map. For each $k \geq 1$ define $C^k_\alpha (A, A)$ to be the space of all multilinear maps $f : A^{\otimes k} \rightarrow A$ satisfying
\begin{center}
$	( \alpha \circ f ) (a_1, \ldots, a_k) = f (\alpha (a_1), \ldots, \alpha (a_k)), ~~ \text{ for all } a_i \in A.$
\end{center}
We define an operad structure on $\mathcal{O} = \{ \mathcal{O}(k) |~ k \geq 1 \}$ where $\mathcal{O}(k) = C^k_\alpha (A, A)$, for $k \geq 1$. Define partial compositions $\circ_i : \mathcal{O}(m) \otimes \mathcal{O}(n) \rightarrow \mathcal{O}(m+n-1)$ by
$$(f \circ_i g)(a_1, \ldots, a_{m+n-1}) = f ( \alpha^{n-1}a_1, \ldots, \alpha^{n-1}a_{i-1}, g (a_i, \ldots, a_{i+n-1}), \alpha^{n-1} a_{i+n}, \ldots, \alpha^{n-1} a_{m+n-1}),$$
for $f \in \mathcal{O}(m),~ g \in \mathcal{O}(n)$ and $a_1, \ldots, a_{m+n-1} \in A$. In view of (\ref{eqn-2}), the compositions
\begin{center}
$\gamma_\alpha : \mathcal{O} (k) \otimes \mathcal{O} (n_1) \otimes \cdots \otimes \mathcal{O} (n_k) \rightarrow \mathcal{O} (n_1 + \cdots + n_k)$
\end{center}
are given by
\begin{align*}
&\gamma_\alpha (f; g_1, \ldots, g_k) (a_1, \ldots, a_{n_1 + \cdots + n_k})\\
&=~ f \big( \alpha^{\sum_{l=2}^{k} |g_l|} g_1 (a_1, \ldots, a_{n_1}), \ldots, ~ \alpha^{\sum_{l=1, l \neq i}^{k} |g_l|} g_i (a_{n_1 + \cdots + n_{i-1} +1}, \ldots, a_{n_1 + \cdots+ n_i}) ,\\
& \qquad \qquad \qquad \qquad \qquad \ldots,~ \alpha^{\sum_{l=1}^{k-1} |g_l|} g_k (a_{n_1 + \cdots + n_{k-1} + 1}, \ldots, a_{n_1 + \cdots + n_k}) \big),
\end{align*}
for $f \in \mathcal{O}(k),~ g_i \in \mathcal{O}(n_i)$ and $a_1, \ldots, a_{n_1 + \cdots + n_k} \in A$.

\begin{prop}\label{hom-ass-operad}
	The partial compositions ~$\circ_i$ (or compositions $\gamma_\alpha$) defines a non-$\sum$ operad structure on $C^\bullet_\alpha (A,A)$ with the identity element given by the identity map {\em id} $\in C^1_\alpha (A, A)$.
\end{prop}

\begin{proof}
For $f \in C^m_\alpha(A, A),~ g \in C^n_\alpha (A, A),~ h \in C^p_\alpha (A,A)$ and $1 \leq i \leq m,~ 1 \leq j \leq n$, we have
\begin{align*}
&((f \circ_i g) \circ_{i+j-1} h)(a_1, \ldots, a_{m+n+p-2}) \\
&= (f \circ_i g) \big(\alpha^{p-1} a_1, \ldots, \alpha^{p-1} a_{i+j-2} ,~ h (a_{i+j-1}, \ldots, a_{i+j+p-2}), \ldots, \alpha^{p-1} a_{m+n+p-2} \big) \\
&= f \big(  \alpha^{n+p-2}a_1, \ldots, \alpha^{n+p-2} a_{i-1}, g \big( \alpha^{p-1}a_i, \ldots, h (a_{i+j-1}, \ldots, a_{i+j+p-2}), \ldots, \alpha^{p-1} a_{i+n+p-2}  \big),\\
& \hspace*{10cm} \ldots, \alpha^{n+p-2} a_{m+n+p-2}    \big) \\
&= f \big(  \alpha^{n+p-2}a_1, \ldots, \alpha^{n+p-2} a_{i-1} , (g \circ_j h) (a_i, \ldots, a_{i+n+p-2}) , \ldots, \alpha^{n+p-2} a_{m+n+p-2} \big) \\
&= (f \circ_i (g \circ_j h)) (a_1, \ldots, a_{m+n+p-2}).
\end{align*}
Similarly, for $1 \leq i < j \leq m$, we have $((f \circ_i g) \circ_{j+n-1} h) = ((f \circ_j h) \circ_i g)$. It is also easy to see that the identity map id is the identity element of the operad. Hence, the proof.
\end{proof}

\begin{remark}
When $\alpha : A \rightarrow A$ is the identity map, one recovers the endomorphism operad on the vector space $A$.
\end{remark}

Note that, the corresponding braces on $C^\bullet_\alpha(A, A)$ are given by
\begin{align*}
\{f\} \{ g_1, \ldots, g_n \} :=& \sum (-1)^\epsilon ~ \gamma_\alpha (f; \text{id}, \ldots, \text{id}, g_1, \text{id}, \ldots, \text{id}, g_n, \text{id}, \ldots, \text{id}).
\end{align*}

Therefore, the degree $-1$ graded Lie bracket on $C^\bullet_\alpha (A, A)$ is given by
\begin{align}\label{grd-lie-ass}
[f , g] = f \circ g - (-1)^{(m-1)(n-1)} g \circ f,
\end{align}
where $$(f \circ g)(a_1, \ldots, a_{m+n-1}) =
\sum_{i=1}^{m} (-1)^{(n-1)(i-1)} f (\alpha^{n-1} a_1, \ldots, g (a_i, \ldots, a_{i+n-1}), \ldots, \alpha^{n-1} a_{m+n-1} ),$$
for $f \in C^m_\alpha (A, A), ~ g \in C^n_\alpha (A, A)$ and $a_1, \ldots, a_{m+n-1} \in A$. See also \cite{amm-ej-makh, das}.

\medskip

Next, let $(A, \mu, \alpha)$ be a hom-associative algebra. Then $\mu \in C^2_\alpha (A, A)$. Moreover, we have
\begin{align*}
\{\mu\} \{ \mu \} (a,b,c) =~& \gamma_\alpha (\mu; \mu, \text{id}) (a,b,c) - \gamma_\alpha (\mu; \text{id}, \mu) (a,b,c) \\
=~& \mu ( \mu(a,b), \alpha (c)) - \mu (\alpha (a), \mu (b,c)) = 0, ~~ \text{ for all } a, b, c \in A.
\end{align*}
Therefore, $\mu$ defines a multiplication on the operad structure on $C^\bullet_\alpha (A,A).$
The corresponding dot product on $C^\bullet_\alpha (A, A)$ is given by
$$ (f \cdot g)(a_1, \ldots, a_{m+n}) = (-1)^{mn} ~\mu (f(\alpha^{n-1}a_1, \ldots, \alpha^{n-1} a_m), g (\alpha^{m-1}a_{m+1}, \ldots, \alpha^{m-1}a_{m+n})  ),$$
for $f \in C^m_\alpha (A, A), ~g \in C^n_\alpha (A,A)$ and $a_1, \ldots, a_{m+n} \in A$. We remark that this dot product on $C^\bullet_\alpha (A,A)$ is same as (up to sign) the cup-product on $C^\bullet_\alpha (A,A)$ defined in \cite{das}. Moreover, the differential $d$ is given by
$$ df = \mu \circ f - (-1)^{|f|} f \circ \mu = (-1)^{|f| + 1} \delta_\alpha (f).$$
The last equality follows from a straightforward calculation \cite{das}.

Thus, in view of Theorem \ref{gers-voro-thm} and Remark \ref{gers-voro-rem}, we get the following.
\begin{thm}
	Let $(A, \mu, \alpha)$ be a hom-associative algebra. Then its Hochschild cochain complex $C^\bullet_\alpha (A, A)$ inherits a homotopy $G$-algebra structure. Hence, its Hochschild cohomology $H^\bullet_\alpha (A, A)$ carries a Gerstenhaber algebra structure.
\end{thm}

\begin{remark}
A direct proof of the existence of a Gerstenhaber algebra structure on the cohomology $H^\bullet_\alpha (A,A)$ has been carried out by the author in \cite{das}. More precisely, the author defined a cup-product 
$\cup_\alpha$ on $C^\bullet_\alpha (A, A)$ by 
\begin{align*}
 (f \cup_\alpha g) (a_1, \ldots, a_{m+n}) =&~ \mu ( f (\alpha^{n-1} a_1, \ldots, \alpha^{n-1} a_m) , g (\alpha^{m-1} a_{m+1}, \ldots, \alpha^{m-1} a_{m+n}) ),
\end{align*}
which is compatible with the Hochschild differential $\delta_\alpha$. Therefore, it induces a cup-product on the cohomology $H^\bullet_\alpha (A,A)$ which turns out to be graded commutative associative. Moreover, the degree $-1$ graded Lie bracket on $C^\bullet_\alpha (A, A)$ as defined in (\ref{grd-lie-ass}) induces a degree $-1$ graded Lie bracket on the cohomology. The induced cup-product and degree $-1$ graded Lie bracket give rise to a (right) Gerstenhaber algebra structure on the cohomology $H^\bullet_\alpha (A, A)$.

The dot product $\cdot$ and the differential $d$ on $C^\bullet_\alpha (A, A)$ induced from the operad structure on $C^\bullet_\alpha (A,A)$ is same as (up to some signs) the cup-product and Hochschild differential on $C^\bullet_\alpha (A,A).$ Due to the presence of signs, we get here (left) Gerstenhaber algebra structure on the cohomology $H^\bullet_\alpha (A,A)$.
\end{remark}

\section{Hom-dialgebras}\label{sec4}

The notion of a (diassociative) dialgebra was introduced by Loday as a generalization of associative algebra and Leibniz algebra \cite{lod}. The hom-analogue of dialgebra is given by the following \cite{yau2}.

\begin{defn}
A hom-dialgebra is a vector space $D$ together with two bilinear maps   $ \dashv, \vdash : D \otimes D \rightarrow D$ and a linear map $\alpha : D \rightarrow D$ satisfying $\alpha (a \dashv b) = \alpha (a) \dashv \alpha (b)$ and $\alpha (a \vdash b) = \alpha (a) \vdash \alpha (b)$  and such that the following axioms hold
\begin{align*}
& \alpha(a) \dashv (b \dashv c) = (a \dashv b) \dashv \alpha(c) = \alpha(a) \dashv (b \vdash c),\\
& (a \vdash b) \dashv \alpha(c) = \alpha(a) \vdash (b \dashv c),\\
& (a \dashv b) \vdash \alpha(c) = \alpha(a) \vdash (b \vdash c) = (a \vdash b ) \vdash \alpha(c), ~~~ \text{ for all } a, b , c \in D.
\end{align*}
\end{defn}
A hom-dialgebra as above is denoted by $(D, \dashv, \vdash, \alpha).$
When $\alpha =$ identity, one get the notion of a dialgebra.
A hom-associative algebra $(A, \mu, \alpha)$ is a hom-dialgebra where $\dashv ~= \mu =~ \vdash$.

Dialgebra cohomology with coefficients was introduced by Frabetti using planar binary trees \cite{frab}. Next, we introduce cohomology of a hom-dialgebra with coefficient in itself. A planar binary tree with $n$-vertices (in short, an $n$-tree) is a planar tree with $(n+1)$ leaves, one root and each vertex trivalent. Let $Y_n$ denote the set of all $n$-trees (see the figure below) and let $Y_0$ be the singleton set consisting\\

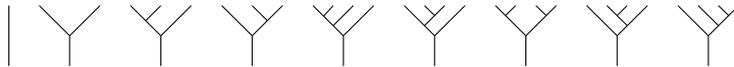
\begin{center}
\begin{tikzpicture}[scale=0.2]
 \draw (-6,0) -- (-4,-2);	\draw (0,0) -- (2,-2);    \draw (6,0) -- (8,-2);  \draw (12,0)-- (14,-2); \draw (14,-2) -- (14,-4); \draw (14, -2) -- (16,0); \draw (12.7, - 0.7) -- (13.3333, 0) ; \draw (13.33, -1.33) -- (14.66, 0);
\draw (-4,-2)-- (-2,0);	\draw (2,-2) -- (4,0);    \draw (8,-2) -- (10,0); \draw (18,0) -- (20,-2); \draw (20, -2) -- (20, -4); \draw (20, -2) -- (22,0); \draw (19.33, -1.33) -- (20.66, 0); \draw (19.33, 0) -- (20, -0.66); \draw (24,0)-- (26,-2); \draw (26,-2)-- (26,-4); \draw (26,-2) -- (28,0); \draw (25.33, 0) -- (24.66, -0.66); \draw (26.66, 0) -- (27.34, -0.66); \draw (30,0) -- (32,-2); \draw (32, -2) -- (32, -4); \draw (32,-2) -- (34, 0); \draw (31.33, 0) -- (32.67 , -1.33) ; \draw (32.66, 0) -- (32, -0.66);
\draw (-4,-2) -- (-4,-4);	\draw (2,-2) -- (2,-4);    \draw (8,-2) -- (8,-4); \draw (36,0) -- (38, -2) ; \draw (38, -2) -- (38, -4) ; \draw (38, -2) -- (40, 0); \draw (37.33, 0) -- ( 38.67, - 1.33) ; \draw (38.66, 0) -- (39.34, -0.66);
\draw (-8,0) -- (-8,-4);	\draw (1,-1) -- (2,0);     \draw (9,-1) -- ( 8,0);
\end{tikzpicture}
\end{center}

\noindent of a root only. Therefore, in the above trees, $Y_0$ consists of the first tree, $Y_1$ consists of the second tree, $Y_2$ consists of the third and the fourth tree, $Y_3$ consists of the rest of the trees.

For each $y \in Y_n$, the $(n+1)$ leaves are labelled by $\{  0, 1, \ldots, n\}$ from left to right, the vertices are labelled by $\{ 1, \ldots, n\}$ so that the $i$-th vertex is between the leaves $(i-1)$ and $i$. The only element in $Y_0$ is denoted by $[0]$ and the only element in $Y_1$ is denoted by $[1]$. The grafting of a $p$-tree $y_1$ and $q$-tree $y_2$ is a $(p+q+1)$-tree denoted by $y_1 \vee y_2$ which is obtained by joining the roots of $y_1$ and $y_2$ and creating a new root from that vertex. This is denoted by $[y_1 \quad p+q+1 \quad y_2]$ with the convention that all zeros are deleted except for the element in $Y_0$. With this notation, the trees in the above (from left to right) are $[0],~[1],~[12],~[21],~[123],~[213],~[131],~[312],~[321].$

For any fixed $n\geq 1$, there are maps $d_i : Y_n \rightarrow Y_{n-1}~ (0 \leq i \leq n)$, $y \mapsto d_i y,$ where $d_iy$ is obtained from $y$ by deleting the $i$-th leaf. These maps are called face maps and satisfy the relations $d_id_j = d_{j-1}d_i$, for all $i < j$.

Before we introduce the cohomology of a hom-dialgebra, we need the following notations. For any $0 \leq i \leq n+1$, the maps $\bullet_i : Y_{n+1} \rightarrow \{ \dashv, \vdash\}$ are defined by

$$ \bullet_0 (y) = \bullet_0^y := \begin{cases} \dashv & \mbox{ if } y \mbox{ is of the form } | \vee y_1 \mbox{for some } n\mbox{-tree } y_1,\\
\vdash & \mbox{ otherwise, }
\end{cases} $$

$$ \bullet_i (y) = \bullet_i^y := \begin{cases} \dashv & \mbox{ if the } i^{th} \mbox{ leaf of } y \mbox{ is oriented like } '\backslash',\\
\vdash & \mbox{ if the } i^{th} \mbox{ leaf of } y \mbox{ is oriented like } '\slash',
\end{cases} $$
for $1 \leq i \leq n$, and
$$ \bullet_{n+1} (y) = \bullet_{n+1}^y := \begin{cases} \vdash & \mbox{ if } y \mbox{ is of the form }  y_1 \vee | , \mbox{for some } n\mbox{-tree } y_1,\\
\dashv & \mbox{ otherwise. }
\end{cases} $$

Let $(D, \dashv, \vdash, \alpha)$ be a hom-dialgebra. For any $n \geq 1$, the cochain group $CY^n_\alpha (D, D)$ consists of all linear maps
\begin{center}
	$ f : K [Y_n] \otimes D^{\otimes n} \rightarrow D, ~ ~ y \otimes a_1 \otimes \cdots \otimes a_n \mapsto f (y; a_1, \ldots, a_n)$
\end{center}
satisfying
\begin{center}
	$(\alpha \circ f) (y; a_1, \ldots, a_n)= f (y ; \alpha (a_1), \ldots, \alpha (a_n))$, ~~~~ for all $y \in Y_n$, $a_i \in D$.
\end{center}
The coboundary map $\delta_\alpha : CY^n_\alpha(D, D) \rightarrow CY^{n+1}_\alpha (D, D)$ defined by
\begin{align*}
(\delta_\alpha f ) (y ; a_1, \ldots, a_{n+1}) =~& \alpha^{n-1} (a_1) \bullet^y_0 f (d_0 y ; a_2, \ldots, a_{n+1}) \\
&+ \sum_{i=1}^{n} (-1)^i~f ( d_iy ;~ \alpha (a_1), \ldots, a_i \bullet_i^y a_{i+1}, \ldots, \alpha (a_{n+1})   ) \\
&+ (-1)^{n+1} f (d_{n+1} y ;~ a_1, \ldots, a_n) \bullet^y_{n+1} \alpha^{n-1} (a_{n+1}),
\end{align*}
for $y \in Y_{n+1}$ and $a_1, \ldots, a_{n+1} \in D.$ Similar to the hom-associative case \cite{amm-ej-makh}, one can prove the following.

\begin{prop}
	The coboundary map satisfies $\delta_\alpha^2 = 0$.
\end{prop}

The cohomology of the complex $(CY^\bullet_\alpha (D,D), \delta_\alpha)$ is called the cohomology of the hom-dialgebra $(D, \dashv, \vdash, \alpha)$ and the cohomology groups are denoted by $HY^n_\alpha (D, D)$, for $n \geq 2$.

\begin{remark}
When $(D, \dashv, \vdash , \alpha)$ is a dialgebra, that is, $\alpha = id$, one recovers the known dialgebra cohomology \cite{lod}. When $(D, \dashv, \vdash , \alpha)$ is a hom-associative algebra, that is, $\dashv ~= \mu =~ \vdash $, one recovers the cohomology of hom-associative algebra.
\end{remark}

We show that the cochain groups $CY^\bullet_\alpha (D,D)$ carries a homotopy $G$-algebra structure. Hence, the cohomology $HY^\bullet_\alpha (D,D)$ inherits a Gerstenhaber algebra structure.

\medskip

\noindent {\bf Operad structure:} Let $D$ be a vector space and $\alpha : D \rightarrow D$ be a linear map. For each $k \geq 1$ define $CY^k_\alpha (D,D)$ to be the space of all multilinear maps $f : K[Y_k] \otimes D^{\otimes k} \rightarrow D$ satisfying
$$ (\alpha \circ f) (y; a_1, \ldots, a_k) = f (y ; \alpha(a_1), \ldots, \alpha (a_k)), ~~ \text{ for all } y \in Y_k \text{ and } a_i \in D.$$
Our aim is to define an operad structure on $\mathcal{O} = \{  \mathcal{O}(k)|~ k \geq 1 \}$ where $\mathcal{O}(k) = CY^k_\alpha (D,D)$, for $k \geq 1$. For this, we closely follow \cite{maj-muk}. 

For any $k, n_1, \ldots, n_k \geq 1$, we define maps
$\mathcal{R}_0 (k; n_1, \ldots, n_k) : Y_{n_1 + \cdots + n_k} \rightarrow Y_k$
by
$$\mathcal{R}_0 (k; n_1, \ldots, n_k) := d_1 \cdots d_{n_1 -1} d_{n_1 +1} \cdots d_{n_1 + n_2 -1} d_{n_1 + n_2 + 1} \cdots d_{n_1 + \cdots + n_{k-1} -1} d_{n_1 + \cdots + n_{k-1} + 1} \cdots d_{n_1 + \cdots n_k -1}.$$
Moreover, for any $1 \leq i \leq k$, there are maps
$ \mathcal{R}_i (k; n_1, \ldots, n_k) : Y_{n_1 + \cdots + n_k} \rightarrow Y_{n_i}$
defined by
$$\mathcal{R}_i (k; n_1, \ldots, n_k) := d_0 d_1 \cdots d_{n_1 + \cdots + n_{i-1} -1}  d_{n_1 + \cdots + n_{i} + 1} \cdots d_{n_1 + \cdots + n_k}.$$
In other words, the function $\mathcal{R}_0 (k; n_1, \ldots, n_k)$ misses $d_0,~ d_{n_1},~ d_{n_1 + n_2}, \ldots, d_{n_1 + \cdots + n_k}$ and the function $\mathcal{R}_i (k; n_1, \ldots, n_k)$ misses $d_{n_1 + \cdots + n_{i-1}},~ d_{n_1 + \cdots + n_{i-1} + 1}, \ldots, d_{n_1 + \cdots + n_i}.$

Then the collection
$$\mathcal{R} = \{ \mathcal{R}_0 (k; n_1, \ldots, n_k) ,~ \mathcal{R}_i (k; n_1, \ldots, n_k) |~ k, n_1, \ldots, n_k \geq 1 \text{ and } 1 \leq i \leq k  \}$$
satisfy the following relations of a pre-operadic system \cite{yau}:

\medskip

\noindent $\bullet$ $\mathcal{R}_0 (k; \underbrace{1, \ldots, 1}_{k \text{ times}}) = id_{Y_k}, ~~ \text{ for each } k \geq 1,$

\medskip

\noindent $\bullet$ $\mathcal{R}_0 (k; n_1, \ldots, n_k) \mathcal{R}_0 (n_1 + \cdots + n_k; ~m_1, \ldots, m_{n_1 + \cdots + n_k}) \\
\hspace*{1cm} = ~\mathcal{R}_0 (k; ~ m_1 + \cdots + m_{n_1}, ~ m_{n_1 + 1} + \cdots + m_{n_1 + n_2}, \ldots, ~m_{n_1 + \cdots + n_{k-1} + 1} + \cdots + m_{n_1 + \cdots + n_k}),$

\medskip

\noindent $\bullet$ $\mathcal{R}_i (k; n_1, \ldots, n_k) \mathcal{R}_0 (n_1 + \cdots + n_k; ~m_1, \ldots, m_{n_1 + \cdots + n_k}) 
= \mathcal{R}_0 ( n_i; m_{n_1 + \cdots + n_{i-1} + 1}, \ldots, m_{n_1 + \cdots + n_i}) \\
\hspace*{1cm} \mathcal{R}_i (k; ~ m_1 + \cdots + m_{n_1}, ~ m_{n_1 + 1} + \cdots + m_{n_1 + n_2}~, \ldots, m_{n_1 + \cdots + n_{k-1} + 1} + \cdots + m_{n_1 + \cdots + n_k}),$

\medskip

\noindent $\bullet$ $\mathcal{R}_{n_1 + \cdots + n_{i-1} + j } (n_1 + \cdots + n_k;~ m_1 , \ldots, m_{n_1 + \cdots + n_k}) = \mathcal{R}_j (n_i; m_{n_1+ \cdots + n_{i-1} + 1}, \ldots, m_{n_1 + \cdots + n_i}) \\
\hspace*{1cm} \mathcal{R}_i (k; ~ m_1 + \cdots + m_{n_1}, ~ m_{n_1 + 1} + \cdots + m_{n_1 + n_2}~, \ldots, m_{n_1 + \cdots + n_{k-1} + 1} + \cdots + m_{n_1 + \cdots + n_k}),$

\medskip

\noindent for any $m_1, \ldots, m_{n_1 + \cdots + n_k} \geq 1$.

Now, we are in a position to define an operad structure on $\mathcal{O}$. Define partial compositions $\circ_i : \mathcal{O}(m) \otimes \mathcal{O}(n) \rightarrow \mathcal{O}(m+n-1)$ by
\begin{align*}
(f \circ_i g) (y; a_1, &\ldots, a_{m+n-1}) = f \bigg( \mathcal{R}_0 (m; \overbrace{1, \ldots, 1, \underbrace{n}_{i\text{-th place}}, 1, \ldots, 1}^{m\text{-tuple}})y ;~ \alpha^{n-1}a_1, \ldots, \alpha^{n-1}a_{i-1},\\ 
& g \big( \mathcal{R}_i (m; \overbrace{1, \ldots, 1, \underbrace{n}_{i\text{-th place}}, 1, \ldots, 1}^{m\text{-tuple}})y ;~ a_i, \ldots, a_{i+n-1} \big), \alpha^{n-1}a_{i+n}, \ldots, \alpha^{n-1} a_{m+n-1}   \bigg),
\end{align*}
for $f \in CY^m_\alpha (D,D),~ g \in CY^n_\alpha (D,D), ~ y \in Y_{m+n-1}$ and $a_1, \ldots, a_{m+n-1} \in D.$ Therefore, by using (\ref{eqn-2}) and the pre-operadic identities, it follows that the compositions
$$\gamma_\alpha : \mathcal{O}(k) \otimes \mathcal{O}(n_1) \otimes \cdots \otimes \mathcal{O}(n_k) \rightarrow \mathcal{O}(n_1 + \cdots + n_k)$$
are given by
\begin{align*}
 \gamma_\alpha (f ; g_1 , \ldots , g_k) (y ; a_1, \ldots,& a_{n_1 + \cdots + n_k}) \\
 = f \big(   \mathcal{R}_0 (k; n_1, \ldots, n_k) y ;~ 
& \alpha^{\sum_{l=2}^{k} |g_l|}~ g_1 \big( \mathcal{R}_1 (k; n_1, \ldots, n_k)y ;~ a_1, \ldots, a_{n_1} \big), \ldots,\\
&  \alpha^{\sum_{l=1, l \neq i}^{k} |g_l|}~ g_i \big( \mathcal{R}_i (k; n_1, \ldots, n_k)y ;~ a_{n_1 + \cdots + n_{i-1} + 1}, \ldots, a_{n_1 + \cdots + n_i} \big), \ldots,\\
&   \alpha^{\sum_{l=1}^{k-1} |g_l|}~ g_k \big( \mathcal{R}_k (k; n_1, \ldots, n_k)y ;~ a_{n_1 + \cdots + n_{k-1} + 1}, \ldots, a_{n_1 + \ldots + n_k}\big)  \big),
\end{align*}
for all $y \in Y_{n_1 + \cdots + n_k}$ and $a_1, a_2, \ldots, a_{n_1 + \cdots + n_k} \in D.$

We also consider the identity map id $\in CY^1_\alpha (D,D)$ defined by  id$ ([1]; a) = a$, for all $a \in D$.

Using the pre-operadic identities of $\mathcal{R}$, we can prove the following. The proof is similar to Proposition \ref{hom-ass-operad}, hence, we omit the details.

\begin{prop}
The partial compositions $\circ_i$ (or compositions $\gamma_\alpha$) defines a non-$\sum$ operad structure on $CY^\bullet_\alpha (D, D)$ with the identity element given by the identity map {\em id} $\in CY^1_\alpha (D,D)$.
\end{prop}

\begin{remark}
When $\alpha : D \rightarrow D$ is the identity map (dialgebra case), one recovers the operad considered in \cite{maj-muk}.
\end{remark}

Note that, the corresponding braces are given by
\begin{align*}
\{f\} \{ g_1, \ldots, g_n \} :=& \sum (-1)^\epsilon ~ \gamma_\alpha (f; \text{id}, \ldots, \text{id}, g_1, \text{id}, \ldots, \text{id}, g_n, \text{id}, \ldots, \text{id}).
\end{align*}

The degree $-1$ graded Lie bracket on $CY^\bullet_\alpha (D,D)$ is given by
$$[f,g] = f \circ g - (-1)^{(m-1)(n-1)} g \circ f,$$
where
\begin{align*}
&(f \circ g) (y; a_1, \ldots, a_{m+n-1})\\
&= \sum_{i=1}^{m} (-1)^{(i-1)(n-1)} f \bigg( \mathcal{R}_0 (m; \overbrace{1, \ldots, 1, \underbrace{n}_{i\text{-th place}}, 1, \ldots, 1}^{m\text{-tuple}})y ;~ \alpha^{n-1}a_1, \ldots, \alpha^{n-1}a_{i-1},\\ 
&\qquad \qquad \qquad g \big( \mathcal{R}_i (m; \overbrace{1, \ldots, 1, \underbrace{n}_{i\text{-th place}}, 1, \ldots, 1}^{m\text{-tuple}})y ;~ a_i, \ldots, a_{i+n-1} \big), \alpha^{n-1}a_{i+n}, \ldots, \alpha^{n-1} a_{m+n-1}   \bigg),
\end{align*}
for $f \in CY^m_\alpha (D,D),~ g \in CY^n_\alpha (D,D)$ and $a_1, \ldots, a_{m+n-1} \in D.$

\medskip

Next, let $(D, \dashv, \vdash, \alpha)$ be a hom-dialgebra. Consider the operad structure on $CY^\bullet_\alpha (D,D)$ as defined above. Define an element $\pi \in CY^2_\alpha (D, D)$ by the following
$$ \pi (y; a, b) := \begin{cases} a \dashv b & \mbox{ if } y = [21],\\
a \vdash b & \mbox{ if } y = [12],
\end{cases} $$
for all $a, b \in D$. An easy calculation shows that
\begin{align*}
 \{ \pi \} \{ \pi \} (y; a, b, c) = \begin{cases} (a \vdash b) \vdash \alpha (c) ~-~ \alpha(a) \vdash (b \vdash c) & \mbox{ if } y = [123],\\
(a \dashv b) \vdash \alpha (c) ~-~ \alpha (a) \vdash (b \vdash c) & \mbox{ if } y = [213], \\
 (a \vdash b) \dashv \alpha (c) ~-~ \alpha (a) \vdash (b \dashv c)  & \mbox{ if } y = [131], \\
 (a \dashv b) \dashv \alpha (c) ~-~ \alpha (a) \dashv (b \vdash c)  & \mbox{ if } y = [312], \\
 (a \dashv b) \dashv \alpha (c) ~-~ \alpha (a) \dashv (b \dashv c)  & \mbox{ if } y = [321].\\
\end{cases}
\end{align*}

Hence, it follows from the hom-dialgebra condition that $ \{ \pi \} \{ \pi \} (y; a, b, c) = 0$, for all $y \in Y_3$ and $a, b, c \in D$.
Therefore, $\pi$ defines a multiplication on the operad $CY^\bullet_\alpha (D,D)$.
The corresponding dot product on $CY^\bullet_\alpha(D,D)$ is given by
\begin{align*}
&(f \cdot g)(y; a_1, \ldots, a_{m+n}) = (-1)^{m}~ \{\pi\} \{f,g\} (y; a_1, \ldots, a_{m+n}) \\
& = (-1)^{mn} ~\pi \big( \mathcal{R}_0 (2;m,n) y ; \alpha^{n-1} f (\mathcal{R}_1 (2;m,n) y ; a_1, \ldots, a_m) , \alpha^{m-1} g (\mathcal{R}_2 (2;m,n) y; a_{m+1}, \ldots, a_{m+n}) \big),
\end{align*}
for $f \in CY^m_\alpha (D,D),~ g \in CY^n_\alpha (D,D), ~ y \in Y_{m+n}$ and $a_1, \ldots, a_{m+n} \in D.$
Like hom-associative case, the differential here is given by
$$df = (-1)^{|f|+1} \delta_\alpha (f).$$

Thus, in view of Theorem \ref{gers-voro-thm} and Remark \ref{gers-voro-rem}, we get the following.
\begin{thm}
	Let $(D, \dashv, \vdash , \alpha)$ be a hom-dialgebra. Then the cochain complex $CY^\bullet_\alpha (D, D)$ inherits a homotopy $G$-algebra structure. Hence, its cohomology $HY^\bullet_\alpha (D, D)$ carries a Gerstenhaber algebra structure.
\end{thm}

\medskip

\noindent {\bf Acknowledgement.} The author would like to thank the referee for his/her comments on the earlier version of the manuscript. The research was supported by the Institute post-doctoral fellowship of Indian Statistical Institute Kolkata. The author would like to thank the Institute for their support.

\end{document}